\documentclass[reqno, 11pt]{amsart}
\usepackage{amsmath, amsthm, amssymb, amsfonts, amscd, latexsym, graphics}

\usepackage{graphicx}
\usepackage{color}
\usepackage{enumerate}
\usepackage{multicol}
\usepackage{array}
\usepackage[all]{xy}
\usepackage{geometry}
\usepackage{enumerate}   

\makeatletter
\@namedef{subjclassname@2020}{%
  \textup{2020} Mathematics Subject Classification}
\makeatother

\def\al{\alpha}
\def\be{\beta}

\def\de{\delta}

\def\La{\Lambda}

\def\ta{\tau}

\def\cG{{\mathcal G}}

\def\cQ{{\mathcal Q}}

\def\ang#1{{\langle #1 \rangle}}

\def\P{{\mathbb P}}

\def\Z{{\mathbb Z}}
\def\N{{\mathbb N}}

\def\det{\operatorname{det}}
\def\dim{\operatorname{dim}}

\def\Ext{\operatorname{Ext}}

\def\gcd{\operatorname{gcd}}

\def\gldim{\operatorname{gldim}}

\def\HH{\operatorname{HH}}
\def\Hom{\operatorname{Hom}}

\def\Im{\operatorname{Im}}

\def\Ker{\operatorname{Ker}}

\def\op{\operatorname{op}}

\def\pd{\operatorname{pd}}

\def\rank{\operatorname{rank}}
\def\rk{\operatorname{rk}}

\def\tr{\operatorname{tr}}

\def\tails{\mathsf{tails}\,}

\def\coh{\mathsf{coh}\,}

\def\Db{\mathsf{D^b}}

\def\mod{\mathsf{mod}\,}

\theoremstyle{plain} 
\newtheorem{thm}{Theorem}[section]
\newtheorem{cor}[thm]{Corollary}
\newtheorem*{thm*}{Theorem}
\newtheorem*{cor*}{Corollary}
\newtheorem{lem}[thm]{Lemma}
\newtheorem{prop}[thm]{Proposition}

\theoremstyle{definition}

\newtheorem{rem}[thm]{Remark}

\newcommand{\thmref}[1]{Theorem~\ref{#1}}
\newcommand{\lemref}[1]{Lemma~\ref{#1}}

\newcommand{\propref}[1]{Proposition~\ref{#1}}

\newcommand{\remref}[1]{Remark~\ref{#1}}

\numberwithin{equation}{section}

\usepackage{multirow,bigdelim,arydshln}   

\title[Hochschild cohomology related to graded down-up algebras]
{Hochschild cohomology related to graded down-up algebras with weights $(1,n)$}
\author{Ayako Itaba}
\address{
Department of Mathematics, 
Faculty of Science, Tokyo University of Science,
1-3 Kagurazaka, Shinjyuku, Tokyo, 162-8601, Japan}
\email{itaba@rs.tus.ac.jp}

\author{Kenta Ueyama}
\address{
Department of Mathematics, 
Faculty of Education, Hirosaki University, 
1 Bunkyocho, Hirosaki, Aomori, 036-8560, Japan}
\email{k-ueyama@hirosaki-u.ac.jp} 
\begin{document}
\begin{abstract}
Let $A=A(\alpha, \beta)$ be a graded down-up algebra with $(\deg x, \deg y)=(1,n)$ and $\beta \neq 0$,
and let $\nabla A$ be the Beilinson algebra of $A$.
If $n=1$, then a description of the Hochschild cohomology group of $\nabla A$ is known.
In this paper, we calculate the Hochschild cohomology group of $\nabla A$ for the case $n \geq 2$.
As an application, we see that the structure of the bounded derived category of the noncommutative projective scheme of $A$
is different depending on whether
$\left(\begin{smallmatrix} 1&0 \end{smallmatrix}\right)\left(\begin{smallmatrix} \alpha &1 \\ \beta &0 \end{smallmatrix}\right)^n\left(\begin{smallmatrix} 1 \\ 0 \end{smallmatrix}\right)$
is zero or not.
Moreover, it turns out that there is a difference between the cases $n=2$ and $n\geq 3$
in the context of Grothendieck groups.

\end{abstract}
\maketitle
\section{Introduction}

Throughout let $k$ be an algebraically closed field of characteristic $0$.
A graded algebra
\[ A(\al, \be) := k\ang{x,y}/(x^2y-\be yx^2 -\al xyx,\; xy^2-\be y^2x -\al yxy) \quad \deg x=m, \deg y=n \in \N^+ \]
with parameters $\al, \be \in k$ is called a graded down-up algebra.
Down-up algebras were originally introduced by Benkart and Roby \cite{BR}
in the study of the down and up operators on partially ordered sets.
Since then, various aspects of these algebras have been investigated;
for example, structures \cite{BW}, \cite{KMP}, \cite{Zh}, representations \cite{CM}, homological invariants \cite{CHS},
connections with enveloping algebras of Lie algebras \cite{B}, \cite{BR}, invariant theory \cite{KK}, \cite{KKZ}, and so on.
In particular, from the viewpoint of noncommutative projective geometry, the following property is of importance.

\begin{thm}[\cite{KMP}] \label{thm:as}
Let $A=A(\al, \be)$ be a graded down-up algebra.
Then $A$ is a noetherian AS-regular algebra of dimension $3$ if and only if $\be \neq 0$.
\end{thm}

A graded down-up algebra has played a key role as a test case for more complicated situations in noncommutative projective geometry.

Let $A= \bigoplus_{i \in \N} A_i=A(\al, \be) $ be a graded down-up algebra with $\be \neq 0$, so that $A$ is AS-regular.
Let $\ell = 2(\deg x+ \deg y)= 2(m+n)$, so that $\ell$ is the Gorenstein parameter of $A$.
Then the Beilinson algebra of $A$ is defined by
\[ \nabla A :=
\begin{pmatrix} A_0 & A_1 & \cdots & A_{\ell-1} \\
0 & A_0 & \cdots & A_{\ell-2} \\
\vdots & \vdots & \ddots & \vdots \\
0 & 0 & \cdots & A_0
\end{pmatrix}
\]
with the multiplication $(a_{ij})(b_{ij})=\left( \sum _{k=0}^{\ell-1}a_{kj}b_{ik} \right)$.
For example, if $\deg x=1, \deg y=1$, then $\nabla A$ is isomorphic to the quotient of the path algebra of the quiver
\begin{align*}
\xymatrix@C=4pc@R=2pc{
1 \ar@<0.5ex>[r]^(0.35){x_1} \ar@<-0.5ex>[r]_(0.65){y_1}  
&2 \ar@<0.5ex>[r]^(0.35){x_2} \ar@<-0.5ex>[r]_(0.65){y_2}
&3 \ar@<0.5ex>[r]^(0.35){x_3} \ar@<-0.5ex>[r]_(0.65){y_3} 
&4
}
\end{align*}
modulo the ideal generated by the relations
\begin{align*}
x_1x_2y_3-\be y_1x_2x_3 -\al x_1y_2x_3,\quad x_1y_2y_3-\be y_1y_2x_3 -\al y_1x_2y_3.
\end{align*}
Let $\tails A$ denote the quotient category of finitely generated graded right $A$-modules by the Serre subcategory of finite dimensional modules, 
and $\mod \nabla A$ the category of finitely generated right $\nabla A$-modules. 
Note that $\tails A$ is considered as the category of coherent sheaves on the noncommutative projective scheme associated to $A$ (see \cite{AZ}).
We write $\Db(\tails A)$ and $\Db(\mod \nabla A)$ for the bounded derived categories
of $\tails A$ and $\mod \nabla A$, respectively. 
The following is obtained as a special case of Minamoto-Mori's theorem \cite[Theorem 4.14]{MM}.

\begin{thm} \label{thm:MM}
If $A=A(\al, \be)$ is a graded down-up algebra with $\be \neq 0$, then $\nabla A$ is extremely Fano of global dimension $2$,
and there exists an equivalence of triangulated categories 
\[\Db(\tails A) \cong \Db(\mod \nabla A).\] 
\end{thm}
We remark that a Fano algebra was renamed as an $n$-representation infinite algebra in \cite{HIO}
from the viewpoint of higher dimensional Auslander-Reiten theory.
By \thmref{thm:MM}, the Beilinson algebras of down-up algebras are important not only in noncommutative projective geometry
but also in representation theory of finite dimensional algebras.

Our interest here is to study the Hochschild cohomology $\HH^i(\nabla A)$ of the Beilinson algebra $\nabla A$ of a down-up algebra $A$.
It is known that the Hochschild cohomology of the Beilinson algebra of an AS-regular algebra $A$ is closely related to
the Hochschild cohomology of $\tails A$ and the infinitesimal deformation theory of $\tails A$ (see \cite{Bel}, \cite{LV2}, \cite{LV1}).
In \cite{Bel}, Belmans computed the Hochschild cohomology of noncommutative planes and noncommutative quadrics,
or in other words, the Hochschild cohomology of the Beilinson algebras of 3-dimensional quadratic or cubic AS-regular $\Z$-algebras.
It should be noted that, to describe the Hochschild cohomology, he used a geometric technique based on the classification of the point schemes of 3-dimensional AS-regular algebras.
Since the point schemes of down-up algebras are divided into three cases
($\P^1 \times \P^1$, a double curve of bidegree $(1,1)$, or two curves of bidegree $(1,1)$ in general position),
\cite[Table 2]{Bel} implies the following result.

\begin{thm}\label{thm:Bel}
Let $A=A(\al, \be)$ be a graded down-up algebra with $\deg x= \deg y=1$ and $\be \neq 0$.
Then
\begin{itemize}
\item $\dim_k \HH^0(\nabla A)=1;$
\item 
$\dim_k \HH^1(\nabla A)=
\begin{cases}
6 \quad \text{if}\ \al=0,\\
3 \quad \text{if}\ \al\neq 0\ \text{and}\ \al^2+4\be=0,\\
1 \quad \text{if}\ \al\neq 0\ \text{and}\ \al^2+4\be\neq 0;
\end{cases}$
\item $\dim_k \HH^2(\nabla A)=
\begin{cases}
9 \quad \text{if}\ \al=0,\\
6 \quad \text{if}\ \al\neq 0\ \text{and}\ \al^2+4\be=0,\\
4 \quad \text{if}\ \al\neq 0\ \text{and}\ \al^2+4\be\neq 0;
\end{cases}$
\item $\dim_k \HH^i(\nabla A)= 0$ for $i\geq 3$.
\end{itemize}
\end{thm}

It is natural to ask what happens when $A$ is not generated in degree 1,
i.e., how the structure of $\HH^i(\nabla A)$ depends on the grading of $A$.
(If $A$ is a graded down-up algebra with $\deg x=m,  \deg y=n$ such that $\gcd(m,n)=r$, then
the $r$-th Veronese algebra $A^{(r)}$ is a graded down-up algebra with $\deg x=m/r,  \deg y=n/r$,
and $\HH^i(\nabla A) \cong \HH^i(\nabla A^{(r)})^r$ by \cite[Theorem 9.1.8(1)]{Wei},
so it is enough to consider the case that $\gcd(\deg x, \deg y) = 1$.)
In this paper, we devote to compute $\HH^i(\nabla A)$ when $A$ is a graded down-up algebra with $\deg x=1, \deg y=n\geq 2$
(so that $\gcd(\deg x, \deg y) = 1$ and $\deg y$ is a multiple of $\deg x$).
We will show the following theorem.

\begin{thm} \label{thm:main}
Let $A=A(\al, \be)$ be a graded down-up algebra with $\deg x= 1, \deg y=n\geq 2$, and $\be \neq 0$.
We define
\[ \de_n := 
\begin{pmatrix} 1 &0 \end{pmatrix}\begin{pmatrix} \al &1 \\ \be &0 \end{pmatrix}^n\begin{pmatrix} 1 \\ 0 \end{pmatrix} \in k
\]
{\rm (}e.g. $\de_2=\al^{2}+\be, \de_3 = \al^3+2\al\be, \de_4=\al^4+3\al^2\be+\be^2, \de_5=\al^5+4\al^3\be+3\al\be^2${\rm )}.
Then
\begin{itemize}
\item $\dim_k \HH^0(\nabla A)=1;$
\item 
$\dim_k \HH^1(\nabla A)=
\begin{cases}
4 \quad \text{if}\ n\ \text{is odd and}\ \al =0\ (\text{in this case $\de_n =0$}),\\
3 \quad \text{if}\ n\ \text{is odd}, \al \neq 0, \text{and}\ \de_n=0, \text{or if}\ n\ \text{is even and}\ \de_n=0,\\
2 \quad \text{if}\ \al^2+4\be=0\ (\text{in this case $\de_n \neq 0$}),\\
1 \quad \text{if}\ \de_n\neq 0\ \text{and}\ \al^2+4\be\neq 0;
\end{cases}$
\item $\dim_k \HH^2(\nabla A)=
\begin{cases}
8 \quad \text{if}\ n=2\ \text{and}\ \de_2=0,\\
7 \quad \text{if}\ n=2\ \text{and}\ \al^2+4\be=0\ (\text{in this case $\de_2\neq 0$}),\\
6 \quad \text{if}\ n=2, \de_2\neq 0, \text{and}\ \al^2+4\be\neq 0,\\
n+5 \quad  \text{if}\ n\ \text{is odd and}\ \al =0 \; (\text{in this case $\de_n= 0$}),\\
n+4 \quad \text{if}\ n\ \text{is odd}, \al \neq 0, \text{and}\ \de_n=0, \text{or if}\ n\geq 4\ \text{is even and}\ \de_n=0, \\
n+3 \quad \text{if}\ n\geq 3\ \text{and}\ \al^2+4\be=0\ (\text{in this case $\de_n\neq 0$}),\\
n+2 \quad \text{if}\ n\geq 3, \de_n\neq 0, \text{and}\ \al^2+4\be\neq 0;
\end{cases}$
\item $\dim_k \HH^i(\nabla A)= 0$ for $i\geq 3$.
\end{itemize}
\end{thm}

Note that it is crucial for this result that $\deg y$ is a multiple of $\deg x$; see \remref{rem:L2}.
Since $A$ is not generated in degree 1,
the geometric theory of point schemes does not work naively in our case,
so our proof of \thmref{thm:main} is purely algebraic.

It is known that Hochschild cohomology is invariant under derived equivalence.
Using \thmref{thm:MM}, we have the following consequence.

\begin{cor}
Let $A=A(\al, \be)$ and $A'=A(\al', \be')$ be graded down-up algebras with $\deg x= 1, \deg y=n\geq 1$, where $\be \neq 0, \be'\neq 0$.
If
\[
\de_n = \begin{pmatrix} 1 &0 \end{pmatrix}\begin{pmatrix} \al &1 \\ \be &0 \end{pmatrix}^n\begin{pmatrix} 1 \\ 0 \end{pmatrix} =0
\ \ \text{and}\ \
\de'_n = \begin{pmatrix} 1 &0 \end{pmatrix}\begin{pmatrix} \al' &1 \\ \be' &0 \end{pmatrix}^n\begin{pmatrix} 1 \\ 0 \end{pmatrix} \neq 0,
\]
then $\Db(\tails A) \ncong \Db(\tails A')$. 
\end{cor} 

In the last section, we will apply our results to the study of Grothendieck groups.
For a graded down-up algebra $A=A(\al, \be)$ with $\deg x= 1, \deg y=n$, and $\be \neq 0$,
we can observe that if $n=1, 2$, then $\Db(\tails A)$ behaves a bit like a geometric object (a smooth projective surface),
but if $n\geq 3$, then $\Db(\tails A)$  is not equivalent to the derived category of any smooth projective surface (\propref{prop:K}).

\section{Proof of \thmref{thm:main}}

In this section, we present the proof of \thmref{thm:main}.
Throughout this section,
let $A=A(\al, \be)$ be a graded down-up algebra with $\deg x= 1, \deg y=n\geq 2$, and $\be \neq 0$.
Then $\La := \nabla A$ is given as the quotient of the path algebra of the quiver
\begin{align*}
\cQ:=
\xymatrix@C=2pc@R=2pc{
1 \ar@<0.5ex>[r]^(0.35){x_1} \ar@/^{-20pt}/@{->}[rrrr]_{y_{1}} 
&2 \ar@<0.5ex>[r]^(0.35){x_2} \ar@/^{-20pt}/@{->}[rrrr]_{y_{2}} 
&\cdots \ar@<0.5ex>[r]^(0.5){x_{n-1}}  
&n \ar@<0.5ex>[r]^(0.35){x_n} \ar@/^{-20pt}/@{->}[rrrr]_{y_{n}}
&n+1 \ar@<0.5ex>[r]^(0.5){x_{n+1}} \ar@/^{-20pt}/@{->}[rrrr]_{y_{n+1}} 
&n+2 \ar@<0.5ex>[r]^(0.5){x_{n+2}} \ar@/^{-20pt}/@{->}[rrrr]_{y_{n+2}} 
&\cdots \ar@<0.5ex>[r]^(0.5){x_{2n-1}}
&2n \ar@<0.5ex>[r]^(0.35){x_{2n}}
&2n+1 \ar@<0.5ex>[r]^(0.5){x_{2n+1}} 
&2n+2
}
\end{align*}
modulo the ideal generated by the relations
\begin{align*}
f_i &:= x_ix_{i+1}y_{i+2}-\be y_ix_{i+n}x_{i+n+1} -\al x_iy_{i+1}x_{i+n+1} \ \ (1 \leq i\leq n),\\
g &:= x_1y_2y_{n+2}-\be y_1y_{n+1}x_{2n+1} -\al y_1x_{n+1}y_{n+2}.
\end{align*}
Note that the arrows $x_i, y_j$ in $\cQ$ come from the generators $x,y$ of $A$, respectively,
and moreover, the relations $f_i, g$ come from the defining relations $x^2y-\be yx^2 -\al xyx, xy^2-\be y^2x -\al yxy$ of $A$, respectively. 

Let $\La^{\rm e}:=\La^{\op}\otimes \La$ be the enveloping algebra of $\La$.
Then $\La$ is a $\La$-bimodule, or equivalently, a right $\La^{\rm e}$-module.
For $i \geq 0$, the $i$-th Hochschild cohomology group $\HH^{i}(\La)$ of $\La$ is defined by 
\[
 \HH^{i}(\La):=\Ext_{\La^{\rm e}}^{i}(\La,\La).
\] 
It is known that $\HH^{0}(\La)$ coincides with the center of $\La$.
Since $\cQ$ is connected and has no oriented cycles, we have
\begin{align} \label{HH0}
\HH^0(\La)= k
\end{align}
(by a similar argument as in the proof of \cite[Lemma 4.4]{LWZ}).

To compute $\HH^{i}(\La)$ for $i \geq 1$, we first construct a minimal projective resolution of $\La$ as a right $\La^{\rm e}$-module
by using Green-Snashall's method \cite[Section 2]{GS}.
We define the sets $\cG^{i} \subset k\cQ$ for $i=0, 1, 2$ by
\[\cG^{0}:=\{e_{1},\ldots, e_{2n+2}\}, \quad  \cG^{1}:=\{x_{1},\ldots,x_{2n+1},y_{1},\ldots,y_{n+2} \}, \quad \cG^{2}:=\{f_{1},\ldots,f_{n},g \}.\]
We set 
\[ P^{i}:=\bigoplus_{h \in \cG^{i}}\La s(h)\otimes t(h)\La \quad \text{for}\ i=0,1,2\]
where $s(h)$ is the vertex at which $h$ starts (source) and $t(h)$ is the vertex at which $h$ ends (target).
Then each $P^{i}$ is a projective right $\La^{\rm e}$-module.

We define $\partial^{0}:P^{0}\rightarrow \La$ to be the multiplication map, and
$\partial^{1}:P^{1}\rightarrow P^{0}$ to be the right $\La^{\rm e}$-homomorphism determined by 
\begin{align*}
\begin{cases}
s(x_i)\otimes t(x_i) \longmapsto (s(e_i)\otimes t(e_{i}))x_i-x_i(s(e_{i+1})\otimes t(e_{i+1})),
\\
s(y_j)\otimes t(y_j) \longmapsto (s(e_j)\otimes t(e_{j}))y_j-y_j(s(e_{j+n})\otimes t(e_{j+n}))
\end{cases}
\end{align*}
for all $1\leq i\leq 2n+1, 1\leq j\leq n+2$. 
We also define $\partial^{2}:P^{2}\rightarrow P^{1}$ to be the right $\La^{\rm e}$-homomorphism determined by 
\begin{align*}
\begin{cases}
s(f_i)\otimes t(f_i)\longmapsto\\
\qquad (s(x_i)\otimes t(x_i))x_{i+1}y_{i+2}  +  x_i(s(x_{i+1})\otimes t(x_{i+1}))y_{i+2}  + x_ix_{i+1}(s(y_{i+2})\otimes t(y_{i+2}))\\
\qquad  -\be((s(y_{i})\otimes t(y_{i}))x_{i+n}x_{i+n+1}  +  y_i(s(x_{i+n})\otimes t(x_{i+n}))x_{i+n+1}  +  y_ix_{i+n}(s(x_{i+n+1})\otimes t(x_{i+n+1})))\\
\qquad  -\al((s(x_{i})\otimes t(x_{i}))y_{i+1}x_{i+n+1}  +  x_i(s(y_{i+1})\otimes t(y_{i+1}))x_{i+n+1}  +  x_iy_{i+1}(s(x_{i+n+1})\otimes t(x_{i+n+1}))),\\
s(g)\otimes t(g)\longmapsto\\
\qquad (s(x_{1})\otimes t(x_{1}))y_2y_{n+2}  +  x_1(s(y_{2})\otimes t(y_{2}))y_{n+2}  +  x_1y_2(s(y_{n+2})\otimes t(y_{n+2}))\\
\qquad -\be ((s(y_{1})\otimes t(y_{1}))y_{n+1}x_{2n+1}  +  y_1(s(y_{n+1})\otimes t(y_{n+1}))x_{2n+1}  +  y_1y_{n+1}(s(x_{2n+1})\otimes t(x_{2n+1})))\\
\qquad -\al ((s(y_{1})\otimes t(y_{1}))x_{n+1}y_{n+2}  +  y_1(s(x_{n+1})\otimes t(x_{n+1}))y_{n+2}  +  y_1x_{n+1}(s(x_{n+2})\otimes t(x_{n+2})))\\
\end{cases} 
\end{align*}
for all $1\leq i\leq n$.

\begin{lem}
With the above definitions, the sequence
\begin{align} \label{mpr}
0\longrightarrow P^{2}\stackrel{\partial^{2}}{\longrightarrow}
P^{1}\stackrel{\partial^{1}}{\longrightarrow}
P^{0}\stackrel{\partial^{0}}{\longrightarrow} \La_{\La^{\rm e}}
\longrightarrow 0
\end{align}
forms a minimal projective resolution of $\La$ as a right $\La^{\rm e}$-module.
\end{lem}

\begin{proof}
By construction, it follows from \cite[Theorem 2.9]{GS} that
$P^{2}\stackrel{\partial^{2}}{\longrightarrow} P^{1}\stackrel{\partial^{1}}{\longrightarrow} P^{0}\stackrel{\partial^{0}}{\longrightarrow}
\La \longrightarrow 0$ forms part of a minimal projective resolution.
Since $\gldim \La=2$ by \thmref{thm:MM}, we see $\pd_{\La^{\rm e}}\La=2$ by \cite[Lemma 1.5]{H}, so the third term, $P^{3}$, is $0$ as required.
\end{proof}

By applying the functor $ \widehat{-} :=\Hom_{\La^{\rm e}}(-,\La)$ to (\ref{mpr}),
we have the Hochschild complex 
\[
0 \longrightarrow \widehat{P^{0}} \stackrel{\widehat{\partial^{1}}}{\longrightarrow} 
  \widehat{P^{1}} \stackrel{\widehat{\partial^{2}}}{\longrightarrow} 
  \widehat{P^{2}} \longrightarrow 0. 
\]
We next calculate a $k$-basis of $\widehat{P^{i}}$.
For $e_i \in \cG^{0}$, we define the right $\La^{\rm e}$-homomorphism $\ta_{e_i} :P^{0}\rightarrow \La$ by
\[
\ta_{e_i}(s(h)\otimes t(h))=
\begin{cases}
e_i& \text{if }h=e_i,\\
0    & \text{otherwise}
\end{cases}
\]
for $h \in \cG^{0}$.

For $x_i, y_j \in \cG^{1}$, we define the right $\La^{\rm e}$-homomorphisms $\ta_{x_i}, \ta_{y_j}, \ta_{y_j}^{x^n}: P^1 \to \La$ by
{\small
\[
\begin{array}{ll}
\text{$
\ta_{x_i}(s(h)\otimes t(h))=
\begin{cases}
x_i & \text{if }h=x_i,\\
0    & \text{otherwise},
\end{cases}
$} & \text{$
\ta_{y_j}(s(h)\otimes t(h))=
\begin{cases}
y_{j} & \text{if }h=y_j,\\
0    & \text{otherwise}, 
\end{cases}
$}
\\[5mm]
\text{$
\ta_{y_j}^{x^n}(s(h)\otimes t(h))=
\begin{cases}
x_j\cdots x_{j+n-1} & \text{if }h=y_j,\\
0    & \text{otherwise}
\end{cases}
$} &
\end{array}
\]
}for $h \in \cG^{1}$.

For $f_i$ and  $g \in \cG^{2}$, we define the right $\La^{\rm e}$-homomorphisms
\[
\ta_{f_i}^{x^{n+2}}, \ta_{f_i}^{yx^2}, \ta_{f_i}^{xyx},
\ta_{g}^{x^{2n+1}}, \ta_{g}^{yx^{n+1}}, \ta_{g}^{xyx^{n}}, \ta_{g}^{y^2x}, \ta_{g}^{yxy}: P^2  \to \La
\] by
{\small
\[
\begin{array}{ll}
\text{$
\ta_{f_i}^{x^{n+2}}(s(h)\otimes t(h))=
\begin{cases}
x_i\cdots x_{i+n+1} & \text{if }h=f_i,\\
0    & \text{otherwise},
\end{cases}
$} & \text{$
\ta_{f_i}^{yx^2}(s(h)\otimes t(h))=
\begin{cases}
y_{i}x_{i+n}x_{i+n+1} & \text{if }h=f_i,\\
0    & \text{otherwise}, 
\end{cases}
$}
\\[5mm]
\text{$
\ta_{f_i}^{xyx}(s(h)\otimes t(h))=
\begin{cases}
x_iy_{i+1}x_{i+n+1} & \text{if }h=f_i,\\
0    & \text{otherwise},
\end{cases}
$} & \text{$
\ta_{g}^{x^{2n+1}}(s(h)\otimes t(h))=
\begin{cases}
x_{1}\cdots x_{2n+1} & \text{if }h=g,\\
0    & \text{otherwise}, 
\end{cases}
$}
\\[5mm]
\text{$
\ta_{g}^{yx^{n+1}}(s(h)\otimes t(h))=
\begin{cases}
y_1x_{n+1}\cdots x_{2n+1} & \text{if }h=g,\\
0    & \text{otherwise}, 
\end{cases}
$} & \text{$
\ta_{g}^{xyx^{n}}(s(h)\otimes t(h))=
\begin{cases}
x_{1}y_{2}x_{n+2}\cdots x_{2n+1} & \text{if }h=g,\\
0    & \text{otherwise}, 
\end{cases}
$}
\\[5mm]
\text{$
\ta_{g}^{y^2x}(s(h)\otimes t(h))=
\begin{cases}
y_{1}y_{n+1}x_{2n+1} & \text{if }h=g,\\
0    & \text{otherwise}, 
\end{cases}
$} & \text{$
\ta_{g}^{yxy}(s(h)\otimes t(h))=
\begin{cases}
y_{1}x_{n+1}y_{n+2} & \text{if }h=g,\\
0    & \text{otherwise}
\end{cases}
$}
\end{array}
\]
}for $h \in \cG^{2}$.

When $n=2$, for $f_i \in \cG^{2}$, we additionally define the right $\La^{\rm e}$-homomorphism $\ta_{f_i}^{y^2}: P^2  \to \La$ by
\[ \ta_{f_i}^{y^2}(s(h)\otimes t(h))=
\begin{cases}
y_iy_{i+2} & \text{if }h=f_i,\\
0    & \text{otherwise}
\end{cases} \]
for $h \in \cG^{2}$.

\begin{lem} \label{lem:dim}
\begin{enumerate}[{\rm (1)}]
\item $\widehat{P^{0}}$ has a $k$-basis
$\{\ta_{e_{1}},\ldots, \ta_{e_{2n+2}}\}$, so $\dim_k \widehat{P^{0}} = 2n+2$.
\item $\widehat{P^{1}}$ has a $k$-basis
$\{\ta_{x_i}, \ta_{y_j}, \ta_{y_j}^{x^n} \mid 1\leq i\leq 2n+1,\; 1\leq j \leq n+2 \}$, so $\dim_k \widehat{P^{1}} = 4n+5$.
\item If $n=2$, then $\widehat{P^{2}}$ has a $k$-basis
$\{ \ta_{f_i}^{x^{4}}, \ta_{f_i}^{yx^2}, \ta_{f_i}^{xyx}, \ta_{f_i}^{y^2}, \ta_{g}^{x^{5}}, \ta_{g}^{yx^{3}}, \ta_{g}^{xyx^2}, \ta_{g}^{y^2x}, \ta_{g}^{yxy}
\mid 1\leq i\leq 2\}$, so $\dim_k \widehat{P^{2}} = 13$.
\item If $n\geq 3$, then $\widehat{P^{2}}$ has a $k$-basis
$\{ \ta_{f_i}^{x^{n+2}}, \ta_{f_i}^{yx^2}, \ta_{f_i}^{xyx}, \ta_{g}^{x^{2n+1}}, \ta_{g}^{yx^{n+1}}, \ta_{g}^{xyx^{n}}, \ta_{g}^{y^2x}, \ta_{g}^{yxy}
\mid 1\leq i\leq n\}$, so $\dim_k \widehat{P^{2}} = 3n+5$.
\end{enumerate}
\end{lem}

\begin{proof}
Since $\widehat{P^{i}} \cong \bigoplus_{h\in {\cG}^{i}} \Hom_{\La^{\rm e}}(\La s(h)\otimes t(h)\La, \La) \cong \bigoplus_{h \in\cG^{i}} s(h)\La t(h)$
as $k$-vector spaces,
if $\ta \in \widehat{P^{i}}$, then $\ta$ is given as $\sum_{h\in {\cG}^{i}} {\ta_h}$,
and $\ta_h(s(h)\otimes t(h))$ is described by a linear combination of basis elements of $s(h)\La t(h)$, so the result follows.
\end{proof}

We then compute a matrix presentation of $\widehat{\partial^{2}}$.
For $i \geq 1$, we define $a_i, b_i \in k$ by
\begin{align*}
\begin{pmatrix} a_i \\ b_i \end{pmatrix}
= \begin{pmatrix} \al &1 \\ \be &0 \end{pmatrix}^{i-1} \begin{pmatrix} 1 \\ 0 \end{pmatrix}.
\end{align*}

\begin{lem} \label{lem:seq}
For any $i\geq 1$, the equality
\begin{align*}
x^{i}y= b_i yx^{i} + a_ixyx^{i-1}
\end{align*}
holds in $A=A(\al,\be)$.
\end{lem}

\begin{proof}
We prove this by induction.
The case $i=1$ is clear. Since
\begin{align*}
x^{i}y&= xx^{i-1}y= x(b_{i-1} yx^{i-1} + a_{i-1}xyx^{i-2})\\
&=b_{i-1}xyx^{i-1} + a_{i-1}(\be yx^2+\al xyx)x^{i-2}\\
&= \be a_{i-1} yx^i + (\al a_{i-1}+b_{i-1})xyx^{i-1},
\end{align*}
it follows that $a_i= \al a_{i-1} +b_{i-1}$ and $b_i = \be a_{i-1}$,
so we get the result.
\end{proof}

Let $L_1$ denote the $(2n+2) \times (3n+3)$ matrix 
{\small
\arraycolsep=2pt
\begin{align*}
&\left(\begin{array}{ccccccccc:ccccccccc:ccccccccccccccc}
\beta &0 &&&&&&&&0&0&&&&&&&& -\beta & -\beta & \beta &0 &&&&&&&& -\beta & \beta \\
\beta & \beta &&&&&&&& -\beta & -\beta & &&&&&&&& -\beta &0 &\beta  \\
\alpha &0 &&&&&&&& -\alpha &0 &&&&&&&&& -\alpha& \alpha &0 \\
 & \beta & \beta & &&&&&&& -\beta& -\beta & &&&&&&&& -\beta &0& \beta \\
 & \alpha &0 &&&&&&&& -\alpha&0 &&&&&&&&& -\alpha & \alpha&0 \\
 &&\beta & \beta & &&&&&&& -\beta& -\beta &&&&&&&&& -\beta &0& \beta \\
 && \alpha &0 &&&&&&&& -\alpha &0 &&&&&&&&& -\alpha & \alpha &0\\
 &&&\ddots&&&&&&&&&\ddots&&&&&&&&&&& \ddots \\
 &&&&& \beta & \beta &&&&&&&& -\beta & -\beta &&&&&&&&& -\beta &0& \beta \\
 &&&&& \alpha &0 &&&&&&&& -\alpha &0 &&&&&&&&& -\alpha & \alpha &0\\
 &&&&&& \ddots&&&&&&&&&\ddots&&&&&&&&&&& \ddots  \\
 &&&&&&&& \beta & \beta & &&&&&&& -\beta & -\beta & &&&&&&&&& -\beta &0& \beta \\
 &&&&&&&& \alpha&0 &&&&&&&& -\alpha&0 &&&&&&&&&& -\alpha & \alpha &0 \\
 &&&&&&&&0 &\alpha& &&&&&&&0& -\alpha& &&&&&&&&&0& -\alpha & \alpha \\
\end{array}\right)\\
&\hspace{3.35truecm} \text{\normalsize \rotatebox{90}{$\Rsh$} the $n$th column} \hspace{1.62truecm} \text{\normalsize \rotatebox{90}{$\Rsh$} the $2n$th column}
\end{align*}
}and let $L_2$ denote the $(n+2) \times (n+2)$ matrix 
{\Large
\[
\left(\begin{smallmatrix}
1&-\alpha &-\beta \\ 
&1&-\alpha &-\beta \\ 
&&1&-\alpha &-\beta \\
&&&&\ddots \\ 
&&&&&1&-\alpha &-\beta \\ 
&&&&&&1&-\alpha &-\beta \\
-\alpha &-\beta &&&&&& b_{n+1} & -(\beta b_n +\alpha b_{n+1}) \\
1& &&&&&& a_{n+1} & -(\beta a_n +\alpha a_{n+1})
\end{smallmatrix}\right)
\]
}where the blank entries represent zeros.

By Lemma \ref{lem:seq}, $-(\be b_n +\al b_{n+1}) = -(\be b_n +\al\be a_n) = -\be a_{n+1}$, so we see
\[
L_2={\Large
\left(\begin{smallmatrix}
1&-\alpha &-\beta \\ 
&1&-\alpha &-\beta \\ 
&&1&-\alpha &-\beta \\
&&&&\ddots \\ 
&&&&&1&-\alpha &-\beta \\ 
&&&&&&1&-\alpha &-\beta \\
-a_2 &-b_2 &&&&&& b_{n+1} & -\be a_{n+1} \\
a_1 &b_1&&&&&& a_{n+1} & -(\beta a_n +\alpha a_{n+1})
\end{smallmatrix}\right)}.
\]

\begin{lem} \label{lem:M2}
{\rm (1)} Assume that $n=2$. Let $\rho_1$ be the ordered basis $\{\ta_{x_1},\dots,\ta_{x_{5}}, \ta_{y_1},\dots,\ta_{y_{4}}, \ta_{y_{4}}^{x^2},\dots,\ta_{y_1}^{x^2}\}$ for $\widehat{P^{1}}$,
and let $\rho_2$ be the ordered basis $\{\ta_{g}^{y^2x}, \ta_{f_1}^{yx^2}, \ta_{g}^{yxy},\ta_{f_2}^{yx^2}, \ta_{f_1}^{xyx}, \ta_{f_2}^{xyx}, \ta_{f_2}^{x^{4}}, \ta_{f_1}^{x^{4}}, \ta_{g}^{yx^{3}}, \ta_{g}^{xyx^{2}}, \ta_{g}^{x^{5}},\ta_{f_1}^{y^{2}}, \ta_{f_2}^{y^{2}} \}$ for $\widehat{P^{2}}$.
Then the matrix representation $M_2$ of $\widehat{\partial^{2}}$ with respect to $\rho_1$ and $\rho_2$ is 
\begin{align*}
\left(
\begin{array}{ccc:ccc}
&& &&&\\
&\mbox{\smash{\Large $L_1$}}& &&&\\
&&  &&& \\ \hdashline
&&&&& \\
&&&&\mbox{\smash{\Large $L_2$}}& \\
&&&&&  \\ \hdashline
0&\cdots&0&0&\cdots&0\\
0&\cdots&0&0&\cdots&0\\
0&\cdots&0&0&\cdots&0
\end{array}
\right).
\end{align*}

{\rm (2)} Assume that $n\geq 3$. Let $\rho_1$ be the ordered basis $\{\ta_{x_1},\dots,\ta_{x_{2n+1}}, \ta_{y_1},\dots,\ta_{y_{n+2}}, \ta_{y_{n+2}}^{x^n},\dots,\ta_{y_1}^{x^n}\}$ for $\widehat{P^{1}}$,
and let $\rho_2$ be the ordered basis
\[
\{\ta_{g}^{y^2x}, \ta_{f_1}^{yx^2}, \ta_{g}^{yxy},\ta_{f_2}^{yx^2}, \ta_{f_1}^{xyx},\dots, \ta_{f_n}^{yx^2}, \ta_{f_{n-1}}^{xyx}, \ta_{f_n}^{xyx}, \ta_{f_n}^{x^{n+2}}, \dots, \ta_{f_1}^{x^{n+2}}, \ta_{g}^{yx^{n+1}}, \ta_{g}^{xyx^{n}}, \ta_{g}^{x^{2n+1}} \}
\]
for $\widehat{P^{2}}$.
Then the matrix representation $M_2$ of $\widehat{\partial^{2}}$ with respect to $\rho_1$ and $\rho_2$ is 
\begin{align} \label{mat}
\left(
\begin{array}{ccc:ccc}
&& &&&\\
&\mbox{\smash{\Large $L_1$}}& &&&\\
&&  &&& \\ \hdashline
&&&&& \\
&&&&\mbox{\smash{\Large $L_2$}}& \\
&&&&&  \\ \hdashline
0&\cdots&0&0&\cdots&0
\end{array}
\right). 
\end{align}
\end{lem}

\begin{proof}
We prove only the last two columns of (\ref{mat}); the others and (1) are similar. 
First, for $h \in \cG^{2}$, we have
\begin{align*}
&\widehat{\partial^{2}}(\ta_{y_2}^{x^n})(s(h)\otimes t(h))
= (\ta_{y_2}^{x^n}\circ \partial^{2})(s(h)\otimes t(h))\\
&= \begin{cases}
\ta_{y_2}^{x^n}(-\be (s(y_2)\otimes t(y_2))x_{n+2}x_{n+3}+\cdots ) &\text{if}\ h=f_2,\\
\ta_{y_2}^{x^n}(-\al x_1(s(y_2)\otimes t(y_2))x_{n+2}+\cdots ) &\text{if}\ h=f_1,\\
\ta_{y_2}^{x^n}(x_1(s(y_2)\otimes t(y_2))y_{n+2}+\cdots ) &\text{if}\ h=g,
\end{cases}\;\;
= \begin{cases}
-\be x_2 \cdots x_{n+1}x_{n+2}x_{n+3} &\text{if}\ h=f_2,\\
-\al x_1x_2 \cdots x_{n+1}x_{n+2} &\text{if}\ h=f_1,\\
x_1x_2 \cdots x_{n+1}y_{n+2} &\text{if}\ h=g. 
\end{cases}
\end{align*}
Since $x_1x_2 \cdots x_{n+1}y_{n+2}= b_{n+1} y_1x_{n+1}\cdots x_{2n+1} + a_{n+1}x_1y_2x_{n+2}\cdots x_{2n+1}$ by \lemref{lem:seq},
\begin{align*}
\widehat{\partial^{2}}(\ta_{y_2}^{x^n})(s(h)\otimes t(h))
&= \begin{cases}
-\be x_2 \cdots x_{n+1}x_{n+2}x_{n+3} &\text{if}\ h=f_2\\
-\al x_1x_2 \cdots x_{n+1}x_{n+2} &\text{if}\ h=f_1\\
b_{n+1} y_1x_{n+1}\cdots x_{2n+1} + a_{n+1}x_1y_2x_{n+2}\cdots x_{2n+1} &\text{if}\ h=g\\
\end{cases}\\
&= (-\be \ta_{f_2}^{x^{n+2}} -\al \ta_{f_1}^{x^{n+2}} +b_{n+1}\ta_{g}^{yx^{n+1}} + a_{n+1}\ta_{g}^{xyx^n})(s(h)\otimes t(h)),
\end{align*}
so we get the $(4n+4)$th column. 
Next, for $h \in \cG^{2}$, we have
\begin{align*}
&\widehat{\partial^{2}}(\ta_{y_1}^{x^n})(s(h)\otimes t(h))
= (\ta_{y_1}^{x^n}\circ \partial^{2})(s(h)\otimes t(h))\\
&= \begin{cases}
\ta_{y_1}^{x^n}(-\be (s(y_1)\otimes t(y_1))x_{n+1}x_{n+2} + \cdots ) &\text{if}\ h=f_1,\\
\ta_{y_1}^{x^n}(-\be (s(y_1)\otimes t(y_1))y_{n+1}x_{2n+1} -\al (s(y_1)\otimes t(y_1))x_{n+1}y_{n+2} + \cdots ) &\text{if}\ h=g,
\end{cases}\\
&= \begin{cases}
-\be x_1 \cdots x_{n}x_{n+1}x_{n+2} &\text{if}\ h=f_1,\\
-\be x_1 \cdots x_{n}y_{n+1}x_{2n+1} -\al x_1 \cdots x_{n}x_{n+1}y_{n+2} &\text{if}\ h=g.
\end{cases}
\end{align*}
It follows from \lemref{lem:seq} that
\begin{align*}
&-\be x_1 \cdots x_{n}y_{n+1}x_{2n+1} -\al x_1 \cdots x_{n}x_{n+1}y_{n+2}\\
&=-\be(b_{n} y_1x_{n+1}\cdots x_{2n} + a_{n}x_1y_2x_{n+2}\cdots x_{2n})x_{2n+1} -\al(b_{n+1} y_1x_{n+1}\cdots x_{2n+1} + a_{n+1}x_1y_2x_{n+2}\cdots x_{2n+1})\\
&=-(\be b_{n}+\al b_{n+1})y_1x_{n+1}\cdots x_{2n+1}-(\be a_n +\al a_{n+1})x_1y_2x_{n+2}\cdots x_{2n+1}.
\end{align*}
Thus we conclude
\[ \widehat{\partial^{2}}(\ta_{y_1}^{x^n})(s(h)\otimes t(h))
= (-\be \ta_{f_1}^{x^{n+2}} -(\be b_{n}+\al b_{n+1})\ta_{g}^{yx^{n+1}} -(\be a_n +\al a_{n+1})\ta_{g}^{xyx^{n}})(s(h)\otimes t(h)). \]
This gives the last column. 
\end{proof}

\begin{rem} \label{rem:L2}
If $\deg x \leq \deg y$ and $\deg y$ is not a multiple of $\deg x$, then there are no paths of length $\deg y$ in the quiver of $\nabla A$, except for $y_i$'s,
so we see that ``the $L_2$ part'' does not appear in $M_2$.
\end{rem}

\begin{lem}\label{lem:de}
Let $\de_n=\left(\begin{smallmatrix} 1 &0 \end{smallmatrix}\right)
\left(\begin{smallmatrix} \al &1 \\ \be &0 \end{smallmatrix}\right)^n
\left(\begin{smallmatrix} 1 \\0 \end{smallmatrix}\right)$.
\begin{enumerate}[{\rm (1)}]
\item If $n$ is odd and $\al= 0$, then $\de_n=0$.
\item If $\al^2+4\be=0$, then $\de_n= (n+1)(\al/2)^n \neq 0$.
\end{enumerate}
\end{lem}

\begin{proof}
(1) 
Clearly, $\de_1=\al=0$. For $n \geq 3$,
\[ \de_n
=\begin{pmatrix} 1 &0 \end{pmatrix}\begin{pmatrix} a_{n+1} \\ b_{n+1} \end{pmatrix}
= a_{n+1} = b_n = \be a_{n-1}
=\be \begin{pmatrix} 1 &0 \end{pmatrix}\begin{pmatrix} a_{n-1} \\ b_{n-1} \end{pmatrix}
=\be \de_{n-2} \]
by Lemma \ref{lem:seq}, so the result follows.

(2) If $\al^2+4\be=0$, then
\[
\begin{pmatrix} \al &1 \\ \be &0 \end{pmatrix}
=  \begin{pmatrix} \al & 1 \\ -\al^2/4 &0 \end{pmatrix}
=  \begin{pmatrix} -1 &1 \\ 1+\al/2 &-\al/2 \end{pmatrix}
   \begin{pmatrix} \al/2 & 0 \\ 1 &\al/2 \end{pmatrix}
   \begin{pmatrix} \al/2 & 1 \\ 1+\al/2 &1 \end{pmatrix},
\]
so it follows that 
\[
\begin{pmatrix} \al &1 \\ \be &0 \end{pmatrix}^n
=\begin{pmatrix} -1 &1 \\ 1+\al/2 &-\al/2 \end{pmatrix}
\begin{pmatrix} (\al/2)^n & 0 \\ n(\al/2)^{n-1} &(\al/2)^n \end{pmatrix}
\begin{pmatrix} \al/2 & 1 \\ 1+\al/2 &1 \end{pmatrix}.
\]
Therefore, we have
\[
\de_n = \begin{pmatrix} -1 &1 \end{pmatrix}
\begin{pmatrix} (\al/2)^n & 0 \\ n(\al/2)^{n-1} &(\al/2)^n \end{pmatrix}
\begin{pmatrix} \al/2 \\ 1+\al/2 \end{pmatrix}
= (n+1)(\al/2)^n.
\]
Since $\be$ is nonzero, so is $\al$ and hence so is $\de_n$. 
\end{proof}

\begin{lem} \label{lem:L1}
The rank of $L_1$ is $n$ if $n$ is odd and $\al= 0$, and it is $n+1$ otherwise.
\end{lem}

\begin{proof}
Since $\be \neq 0$, we have
\begin{align*}
L_1&\to 
{\scriptsize
\arraycolsep=2pt 
\begin{array}{llccccccccccccccccccccccccrr}
\ldelim({13}{1pt} && \beta &0&&&&&0&0&&&&&& -\beta & -\beta & \beta &0&&&&&& -\beta & \beta &\rdelim){13}{1pt} &\rdelim\}{7}{1pt}[\text{$n+1$ rows}] \\
&&\beta & \beta &&&&& -\beta & -\beta & &&&&&& -\beta &0 &\beta  \\
&&& \beta & \beta & &&&& -\beta& -\beta & &&&&&& -\beta &0& \beta \\
&& &&&\ddots&&&&&&&\ddots&&&&&&&&& \ddots \\
&& &&&& \beta & \beta &&&&&& -\beta & -\beta &&&&&&&& -\beta &0& \beta \\
&& &&&&& \beta & \beta & &&&&& -\beta & -\beta &&&&&&&& -\beta &0& \beta \\ \cdashline{3-26}
&&\alpha & &&&&&-\alpha&  & &&&&&& -\alpha& \alpha &&&&&&&&&&\rdelim\}{6}{1pt}[\text{$n+1$ rows}]\\
&& & \alpha &&&&&&-\alpha& & &&&&&& -\alpha & \alpha \\
&& &&&\ddots&&&&&&&\ddots&&&&&&&&&\ddots\\
&& &&&&& \alpha&&&&&&& -\alpha& &&&&&&&& -\alpha & \alpha \\
&& &&&&&& \alpha&&&&&&& -\alpha& &&&&&&&& -\alpha & \alpha \\
\end{array}}
\\
&\to{\scriptsize
\arraycolsep=2pt
\begin{array}{llccccccccccccccccccccccccrr}
\ldelim({13}{1pt} && \beta &0&&&&&0&0&&&&&& -\beta & -\beta & \beta &0&&&&&& -\beta & \beta &\rdelim){13}{1pt} &\rdelim\}{7}{1pt}[\text{$n+1$ rows}] \\
&&\beta & \beta &&&&& -\beta & -\beta & &&&&&& -\beta &0&\beta  \\
&&& \beta & \beta & &&&& -\beta& -\beta & &&&&&& -\beta &0& \beta \\
&& &&&\ddots&&&&&&&\ddots&&&&&&&&& \ddots \\
&& &&&& \beta & \beta &&&&&& -\beta & -\beta &&&&&&&& -\beta &0& \beta \\
&& &&&&& \beta & \beta & &&&&& -\beta & -\beta &&&&&&&& -\beta &0& \beta \\ \cdashline{3-26}
&&\alpha &\alpha&&&&&-\alpha&-\alpha  & &&&&&& -\alpha&0& \alpha &&&&&&&&&\rdelim\}{6}{1pt}[\text{$n+1$ rows}]\\
&& &\alpha &\alpha&&&&&-\alpha&-\alpha  & &&&&&& -\alpha&0& \alpha \\
&& &&&\ddots&&&&&&&\ddots&&&&&&&&&\ddots\\
&& &&&&&\alpha &\alpha&&&&&&-\alpha&-\alpha  &&&&&&&& -\alpha&0& \alpha \\
&& &&&&&& \alpha&&&&&&& -\alpha& &&&&&&&& -\alpha & \alpha \\
\end{array}}
\\
&\to{\scriptsize
\arraycolsep=2pt
\begin{array}{llccccccccccccccccccccccccrr}
\ldelim({12}{1pt} && \beta &0&&&&&0&0&&&&&& -\beta & -\beta & \beta &0&&&&&& -\beta & \beta &\rdelim){12}{1pt} &\rdelim\}{8}{1pt}[\text{$n+2$ rows}] \\
&&\beta & \beta &&&&& -\beta & -\beta & &&&&&& -\beta &0&\beta  \\
&&& \beta & \beta & &&&& -\beta& -\beta & &&&&&& -\beta &0& \beta \\
&& &&&\ddots&&&&&&&\ddots&&&&&&&&& \ddots \\
&& &&&& \beta & \beta &&&&&& -\beta & -\beta &&&&&&&& -\beta &0& \beta \\
&& &&&&& \beta & \beta & &&&&& -\beta & -\beta &&&&&&&& -\beta &0& \beta \\ 
&& &&&&&& \alpha&&&&&&& -\alpha& &&&&&&&& -\alpha & \alpha \\ \cdashline{3-26}
&&0&0&\cdots &&&&&&&&&&&&&&&&&&&\cdots&0&0&&\rdelim\}{4}{1pt}[\text{$n$ rows}]\\
&&\vdots &&&&&&&&&&&&&&&&&&&&&&&\vdots\\
&&0&0&\cdots &&&&&&&&&&&&&&&&&&&\cdots&0&0\\
\end{array}}
\end{align*}
by elementary row operations, so we get
\begin{align*}
&\rank L_1
=\rank {\scriptsize
\arraycolsep=2pt
\begin{array}{llccccccccccccccccccccccccr}
\ldelim({9}{1pt} &&1 & &&&&&0&0&&&&&& -1 & -1 & 1 &&&&&&& -1 & 1 &\rdelim){9}{1pt}\\
&&&1 &&&&& -1 & -1 & &&&&&1& 0 &-1 &1&&&&&& 1 & -1 \\
&&&&1 & &&&1& 0& -1 & &&&&-1&0& 0 &-1& 1 &&&&& -1 & 1 \\
&&&&&\ddots&&&&&&&\ddots&&&&&&&&& \ddots \\
&&&&&& 1 &0&(-1)^{n-2}&&&&  &-1 &0 &(-1)^{n-1}&&&&&&&-1&1  &(-1)^{n-1}& (-1)^{n-2} \\
&&&&&& \beta & \beta &&&&&& -\beta & -\beta &&&&&&&& -\beta &0& \beta \\
&&&&&&& \beta & \beta & &&&&& -\beta & -\beta &&&&&&&& -\beta &0& \beta \\
&&&&&&&& \alpha&&&&&&& -\alpha& &&&&&&&& -\alpha & \alpha 
\end{array}}\\
&=\rank {\scriptsize
\arraycolsep=2pt
\begin{array}{llccccccccccccccccccccccccr}
\ldelim({8}{1pt} &&1 & &&&&&0&0&&&&&& -1 & -1 & 1 &&&&&&& -1 & 1 &\rdelim){8}{1pt}\\
&&&1 &&&&& -1 & -1 & &&&&&1& 0 &-1 &1&&&&&& 1 & -1 \\
&&&&1 & &&&1& 0& -1 & &&&&-1&0& 0 &-1& 1 &&&&& -1 & 1 \\
&&&&&\ddots&&&&&&&\ddots&&&&&&&&& \ddots \\
&&&&&&& 1 &(-1)^{n-1}&&&&&  &-1 &(-1)^{n}&&&&&&&& -1 &1+(-1)^{n}& (-1)^{n-1} \\
&&&&&&& \beta & \beta & &&&&& -\beta & -\beta &&&&&&&& -\beta &0& \beta \\
&&&&&&&& \alpha&&&&&&& -\alpha& &&&&&&&& -\alpha & \alpha \\
\end{array}}\\
&=\rank 
{\scriptsize
\arraycolsep=1.5pt
\begin{array}{lllccccccccccccccccccccccccrr}
\ldelim({8}{1pt} &&&1 & &&&&&0&0&&&&&& -1 & -1 & 1 &&&&&&& -1 & 1 &\rdelim){8}{1pt}&\rdelim\}{8}{1pt}[\text{$n+2$ rows}.]\\
&&&&1 &&&&& -1 & -1 & &&&&&1& 0 &-1 &1&&&&&& 1 & -1 \\
&&&&&1 & &&&1& 0& -1 & &&&&-1&0& 0 &-1& 1 &&&&& -1 & 1 \\
&&&&&&\ddots&&&&&&&\ddots&&&&&&&&& \ddots \\
&&&&&&&& 1 &(-1)^{n-1}&&&&&  &-1 &(-1)^{n}&&&&&&& &-1 &1+(-1)^{n}& (-1)^{n-1} \\
&&&&&&&&& 1+(-1)^{n} & &&&&&  & -1+(-1)^{n+1} &&&&&&&& &-1+(-1)^{n+1}& 1+(-1)^{n} \\
&&&&&&&&& \alpha&&&&&&& -\alpha& &&&&&&&& -\alpha & \alpha \\
\end{array}}.
\end{align*}
Hence the assertion follows.
\end{proof}

\begin{lem} \label{lem:L2}
Let $\de_n=\left(\begin{smallmatrix} 1 &0 \end{smallmatrix}\right)\left(\begin{smallmatrix} \al &1 \\ \be &0 \end{smallmatrix}\right)^n\left(\begin{smallmatrix} 1 \\ 0 \end{smallmatrix}\right)$.
If $\de_n=0$, then $\rank L_2 = n$.
If $\de_n \neq 0$ and $\al^2+4\be= 0$, then $\rank L_2 =n+1$.
If $\de_n \neq 0$ and $\al^2+4\be \neq 0$, then $\rank L_2 = n+2$.
\end{lem}

\begin{proof}
By elementary row operations, we have
\begin{align*}
&L_2\rightarrow 
\left(\begin{smallmatrix}
1&-\alpha &-\beta \\ 
&1&-\alpha &-\beta \\ 
&&1&-\alpha &-\beta \\
&&&&\ddots \\  \\ 
&&&&&1&-\alpha &-\beta \\ 
&&&&&&1&-\alpha &-\beta \\
0 &-b_2-\al a_2 &-\be a_2 &&&&& b_{n+1} & -\be a_{n+1} \\
0 &b_1+\al a_1 &\be a_1 &&&&& a_{n+1} & -(\beta a_n +\alpha a_{n+1})
\end{smallmatrix}\right)
=
\left(\begin{smallmatrix}
1&-\alpha &-\beta \\ 
&1&-\alpha &-\beta \\ 
&&1&-\alpha &-\beta \\
&&&&\ddots \\  \\ 
&&&&&1&-\alpha &-\beta \\ 
&&&&&&1&-\alpha &-\beta \\
0 &-a_3 &-b_3&&&&& b_{n+1} & -\be a_{n+1} \\
0 &a_2 &b_2&&&&& a_{n+1} & -(\beta a_n +\alpha a_{n+1})
\end{smallmatrix}\right)
\\
&\rightarrow \cdots \rightarrow
\left(\begin{smallmatrix}
1&-\alpha &-\beta \\ 
&1&-\alpha &-\beta \\ 
&&&&\ddots \\ 
&&&&&1&-\alpha &-\beta \\ 
&&&&&&1&-\alpha &-\beta \\
0 &&\cdots &&0 &-a_{n} &-b_{n}& b_{n+1} & -\be a_{n+1} \\
0 &&\cdots &&0 &a_{n-1} &b_{n-1}& a_{n+1} & -(\beta a_n +\alpha a_{n+1})
\end{smallmatrix}\right)
\rightarrow 
{\left(\begin{smallmatrix}
1&-\alpha &-\beta \\ 
&1&-\alpha &-\beta \\ 
&&&\ddots \\ 
&&&&1&-\alpha &-\beta \\ 
&&&&&1&-\alpha &-\beta \\
0 &&\cdots &&0 &-a_{n+1}& 0 & -\be a_{n+1} \\
0 &&\cdots &&0 &a_{n}& a_{n+1}+b_n & -(\beta a_n +\alpha a_{n+1})
\end{smallmatrix}\right)}
\\
&\rightarrow
{\left(\begin{smallmatrix}
1&-\alpha &-\beta \\ 
&1&-\alpha &-\beta \\ 
&&&\ddots \\ 
&&&&1&-\alpha &-\beta \\ 
&&&&&1&-\alpha &-\beta \\
0  &&\cdots &&0 &0&-\al a_{n+1} & -2\be a_{n+1} \\
0  &&\cdots &&0 &0 & a_{n+1} +b_n +\al a_n  & -\alpha a_{n+1}
\end{smallmatrix}\right)}
= 
\left(\begin{smallmatrix}
1&-\alpha &-\beta \\ 
&1&-\alpha &-\beta \\ 
&&&&\ddots \\ 
&&&&&1&-\alpha &-\beta \\ 
&&&&&&1&-\alpha &-\beta \\
0 &0 &&\cdots &&0 &0&-\al a_{n+1} & -2\be a_{n+1} \\
0 &0 &&\cdots &&0 &0 & 2a_{n+1}  & -\alpha a_{n+1}
\end{smallmatrix}\right).
\end{align*}
Since
$\det 
\left(\begin{smallmatrix}-\al a_{n+1} & -2\be a_{n+1} \\ 2a_{n+1}  & -\alpha a_{n+1} \end{smallmatrix}\right)
=a_{n+1}^2(\al^2+4\be)$ and 
$a_{n+1}
=\left(\begin{smallmatrix} 1 &0 \end{smallmatrix}\right)\left(\begin{smallmatrix} a_{n+1} \\ b_{n+1} \end{smallmatrix}\right)
=\de_n$,
the result follows.
\end{proof}

We are now ready to prove \thmref{thm:main}.

\begin{proof}[Proof of \thmref{thm:main}]
By (\ref{HH0}), we have $\dim_k \HH^0(\nabla A)=1$, and so $\dim_{k} \Ker \widehat{\partial^{1}}=1$.
It follows from \lemref{lem:dim}(1) that
$\dim_{k}\Im \widehat{\partial^{1}}= \dim_k \widehat{P^{0}} - \dim_{k}\Ker \widehat{\partial^{1}} =(2n+2)-1=2n+1$.
Lemmas \ref{lem:dim}(2) and \ref{lem:M2} imply
\begin{align*}
\dim_k \HH^1(\nabla A)
&=\dim_{k} \Ker \widehat{\partial^{2}}-\dim_{k}\Im \widehat{\partial^{1}}\\
&=(\dim_{k} \widehat{P^{1}}-\dim_{k}\Im \widehat{\partial^{2}})-\dim_{k}\Im \widehat{\partial^{1}}\\
&=\dim_{k}\widehat{P^{1}}-\rank M_{2}-\dim_{k} \Im \widehat{\partial^{1}}\\
&=(4n+5) -\rank L_1-\rank L_2-(2n+1)\\
&=2n+4 -\rank L_1-\rank L_2.
\end{align*}
By Lemmas \ref{lem:de}, \ref{lem:L1}, and \ref{lem:L2}, we obtain
\begin{align*}
&\dim_k \HH^1(\nabla A)\\
&=
\begin{cases}
2n+4 - n -n =4 
&\text{if}\ n\ \text{is odd and}\ \al =0 \ (\text{in this case $\de_n =0$}),\\
2n+4 -(n+1)- n =3
&\text{if}\ n\ \text{is odd}, \al \neq 0, \text{and}\ \de_n=0,\text{or if}\ n\ \text{is even and}\ \de_n=0,\\
2n+4 - (n+1)-(n+1)=2
&\text{if}\ \al^2+4\be=0\ (\text{in this case $\de_n\neq 0$}), \\
2n+4 - (n+1)-(n+2)=1
&\text{if}\ \de_n\neq 0\ \text{and}\ \al^2+4\be\neq 0.\\
\end{cases}
\end{align*}
Moreover, since
\begin{align*}
\dim_k\HH^{2}(\nabla A)
&=\dim_{k}\Ker \widehat{\partial^{3}}-\dim_{k}\Im \widehat{\partial^{2}}\\
&=\dim_{k} \widehat{P^{2}}-\rank M_{2}\\
&=\dim_{k} \widehat{P^{2}} -\rank L_1- \rank L_2\\
&=\dim_{k} \widehat{P^{2}} +\dim_k \HH^1(\nabla A) -(2n+4)\\
&=\begin{cases}
13 +\dim_k \HH^1(\nabla A) -8  & \text{if}\ n=2, \\
3n+5 +\dim_k \HH^1(\nabla A) -(2n+4) & \text{if}\ n\geq 3, 
\end{cases}\\
&=\begin{cases}
\dim_k \HH^1(\nabla A)+5 & \text{if}\ n=2, \\
\dim_k \HH^1(\nabla A)+n+1 & \text{if}\ n\geq 3 
\end{cases}
\end{align*}
by Lemmas \ref{lem:dim}(3),(4), it follows that
\[
\dim_k \HH^2(\nabla A)=
\begin{cases}
8 \quad \text{if}\ n=2\ \text{and}\ \de_2=0,\\
7 \quad \text{if}\ n=2\ \text{and}\ \al^2+4\be=0\ (\text{in this case $\de_2\neq 0$}),\\
6 \quad \text{if}\ n=2, \de_2\neq 0, \text{and}\ \al^2+4\be\neq 0,\\
n+5 \quad  \text{if}\ n\ \text{is odd and}\ \al =0 \; (\text{in this case $\de_n=0$}),\\
n+4 \quad \text{if}\ n\ \text{is odd}, \al \neq 0, \text{and}\ \de_n=0, \text{or if}\ n\geq 4\ \text{is even and}\ \de_n=0, \\
n+3 \quad \text{if}\ n\geq 3\ \text{and}\ \al^2+4\be=0\ (\text{in this case $\de_n\neq 0$}), \\
n+2 \quad \text{if}\ n\geq 3, \de_n\neq 0, \text{and}\ \al^2+4\be\neq 0. 
\end{cases}
\]
Clearly $\HH^i(\nabla A)=0$ for $i\geq 3$, so the proof is completed.
\end{proof}

\section{Discussion on Grothendieck groups}

At the end of the paper, we give a discussion of our results in the context of Grothendieck groups
based on \cite[Section 3]{BP} and \cite[Section 3.1]{B}.
Let $\sf T$ be a triangulated category, and
let $K_0(\sf T)$ be the Grothendieck group of $\sf T$ (see \cite[Section 3]{BP} for details).
If $\sf T$ admits a full strong exceptional sequence of length $r$, then $K_0(\sf T)$ is $\Z^{r}$, so $\rk K_0({\sf T})=r$.
If $\sf T$ has the Serre functor $S$ in the sense of Bondal and Kapranov \cite{BK}, then 
$S$ induces an automorphism $\mathfrak s$ of $K_0({\sf T})$.
\begin{thm} \label{thm:K}
Let $\Db(\coh X)$ be the bounded derived category of coherent sheaves on a smooth projective variety $X$.
\begin{enumerate}[{\rm (1)}]
\item \textnormal{(\cite[Lemma 3.1]{BP})} The action of $(-1)^{\dim X}{\mathfrak s}$ on $K_0(\Db(\coh X))$ is unipotent.
\item \textnormal{(\cite[Corollary 25]{Bel})} If $\Db(\coh X)$ admits a full strong exceptional sequence, then
\[
\chi(\HH^\bullet(X)) = (-1)^{\dim X}\rk K_0(\Db(\coh X)).
\]
where $\chi(\HH^\bullet(X)):= \sum_{i \in \Z}(-1)^i \dim_k \HH^i(X)$.
\end{enumerate}
\end{thm}

Let $A=A(\al, \be)$ be a graded down-up algebra with $\deg x=1, \deg y= n\geq 1$, and $\be \neq 0$.
Then $\Db(\tails A)$ has a full strong exceptional sequence of length $2n+2$ by \cite[Propositions 4.3, 4.4]{MM},
so $\rk K_0(\Db(\tails A)) =2n+2$.
Moreover $\Db(\tails A)$ has the Serre functor by \cite[Appendix A]{dNV}.
Note that $\gldim (\tails A)= \gldim \nabla A=2$.

If $n=1$, then ${\mathfrak s}$ acts unipotently on $K_0(\Db(\tails A))$ (see  \cite[comments after Remark 26]{Bel}), and 
it follows from \thmref{thm:Bel} that 
\[ \chi(\HH^\bullet(\nabla A)) = 4 = \rk K_0(\Db(\tails A))\]
where $\chi(\HH^\bullet(\nabla A)):= \sum_{i \in \Z}(-1)^i \dim_k \HH^i(\nabla A)$, so an analogue of \thmref{thm:K} holds.
For the case $n\geq 2$, we have the following.

\begin{prop}\label{prop:K}
If $n=2$, then ${\mathfrak s}$ acts unipotently on $K_0(\Db(\tails A))$  and  
\begin{equation}\label{eq:E1}
 \chi(\HH^\bullet(\nabla A)) = 6 = \rk K_0(\Db(\tails A)).
\end{equation}
If $n\geq 3$, then ${\mathfrak s}$ does not act unipotently on $K_0(\Db(\tails A))$ and  
\begin{equation}\label{eq:E2}
\chi(\HH^\bullet(\nabla A)) = n+2 \neq 2n+2 = \rk K_0(\Db(\tails A)).
\end{equation}
\end{prop}

\begin{proof}
First, (\ref{eq:E1}) and (\ref{eq:E2}) follow from \thmref{thm:main}. If $n=2$,
then the Gram matrix of $K_0(\Db(\tails A))$ is given by 
\[ M=
\begin{pmatrix}
1 & 1 & 2 & 3 & 4 & 5\\
0 & 1 & 1 & 2 & 3 & 4\\
0 & 0 & 1 & 1 & 2 & 3\\
0 & 0 & 0 & 1 & 1 & 2\\
0 & 0 & 0 & 0 & 1 & 1\\
0 & 0 & 0 & 0 & 0 & 1
\end{pmatrix},\] 
so one can verify that
\[
{\mathfrak s}= M^{-1}M^{T}=
\begin{pmatrix}
7 & 5 & 4 & 3 & 2 & 1\\
1 & 2 & 1 & 1 & 1 & 1\\
-5 & -4 & -2 & -2 & -1 & 0\\
-6 & -5 & -4 & -2 & -2 & -1\\
-1 & -1 & -1 & -1 & 0 & -1\\
5 & 4 & 3 & 2 & 1 & 1
\end{pmatrix}\]
is unipotent.
We now consider the case $n\geq 3$.
Suppose that ${\mathfrak s}$ acts unipotently on $K_0(\Db(\tails A))$.
Then ${\mathfrak s}=M^{-1}M^{T}$ is unipotent
where $M$ is the Gram matrix.
Since $M$ coincides with the Cartan matrix of $\nabla A$, 
the Coxeter matrix $\Phi_{\nabla A}$ of $\nabla A$ is obtained as $-M^{-T}M$ (see \cite[Section 1]{H2}), so we have
\[ -\tr \Phi_{\nabla A} = \tr M^{-T}M = \tr M^{-1}M^{T} = \tr {\mathfrak s}.\]
Unipotency of ${\mathfrak s}$ implies $\tr {\mathfrak s}=2n+2$, so it follows that $-\tr \Phi_{\nabla A}=2n+2$.
By Happel's trace formula \cite{H2}, we have $\chi(\HH^\bullet(\nabla A))=-\tr \Phi_{\nabla A}=2n+2$, but this contradicts (\ref{eq:E2}).
Hence ${\mathfrak s}$ does not act unipotently.
\end{proof}

In respect of \propref{prop:K}, when $n=2$, $\Db(\tails A)$ behaves a bit like a geometric object (a smooth projective surface),
but when $n\geq 3$, $\Db(\tails A)$ is not equivalent to the derived category of any smooth projective surface.

\section*{Acknowledgments}
The authors are grateful to Teruyuki Yorioka for his support and helpful discussions.
They also thank the referee for useful comments in improving the paper.
The first author was supported by JSPS Grant-in-Aid for Early-Career Scientists 18K13397.
The second author was supported by JSPS Grant-in-Aid for Early-Career Scientists 18K13381.

 
\end{document}